\documentclass[11pt]{article}
\usepackage[pagebackref=true,colorlinks,linkcolor=blue,citecolor=magenta]{hyperref}
\usepackage{amssymb,amsmath}
\usepackage{amsmath}
\usepackage{dsfont}
\usepackage{amsfonts}
\newtheorem{precor}{{\bf Corollary}}

\newenvironment{cor}{\begin{precor}{\hspace{-0.5
               em}{\bf.\ }}}{\end{precor}}
\newtheorem{prealphcor}{{\bf Corollary}}

\newtheorem{precon}{{\bf Conjecture}}

\newtheorem{prealphcon}{{\bf Conjecture}}

\newenvironment{alphcon}{\begin{prealphcon}{\hspace{-0.5
               em}{\bf.\ }}}{\end{prealphcon}}
\newtheorem{predefin}{{\bf Definition}}

\newtheorem{preexm}{{\bf Example}}

\newtheorem{preappl}{{\bf Application}}

\newtheorem{prelem}{{\bf Lemma}}

\newenvironment{lem}{\begin{prelem}{\hspace{-0.5
               em}{\bf.\ }}}{\end{prelem}}
\newtheorem{preproof}{{\bf Proof.\ }}

\newenvironment{proof}[1]{\begin{preproof}{\rm
               #1}\hfill{$\blacksquare$}}{\end{preproof}}
\newtheorem{prethm}{{\bf Theorem}}

\newenvironment{thm}{\begin{prethm}{\hspace{-0.5
               em}{\bf.\ }}}{\end{prethm}}
\newtheorem{prealphthm}{{\bf Theorem}}

\newenvironment{alphthm}{\begin{prealphthm}{\hspace{-0.5
               em}{\bf.\ }}}{\end{prealphthm}}

\newtheorem{prealphlem}{{\bf Lemma}}

\newtheorem{prealphprb}{{\bf Problem}}

\newenvironment{alphprb}{\begin{prealphprb}{\hspace{-0.5
               em}{\bf.\ }}}{\end{prealphprb}}

\newtheorem{prepro}{{\bf Proposition}}

\newenvironment{pro}{\begin{prepro}{\hspace{-0.5
               em}{\bf.\ }}}{\end{prepro}}
\newtheorem{preprb}{{\bf Problem}}

\newtheorem{prerem}{{\bf Remark}}

\newtheorem{preapp}{{\bf Application}}

\newtheorem{prequ}{{\bf Question}}

\newenvironment{qu}{\begin{prequ}{\hspace{-0.5
               em}{\bf.\ }}}{\end{prequ}}
\newtheorem{preclaim}{{\bf Claim}}

\def\conct[#1,#2]{\mbox {${#1} \leftrightarrow {#2}$}}
\def\dconct[#1,#2]{\mbox {${#1} \rightarrow {#2}$}}
\def\deg[#1,#2]{\mbox {$d_{_{#1}}(#2)$}}
\def\mindeg[#1]{\mbox {$\delta_{_{#1}}$}}
\def\maxdeg[#1]{\mbox {$\Delta_{_{#1}}$}}
\def\outdeg[#1,#2]{\mbox {$d_{_{#1}}^{^+}(#2)$}}
\def\minoutdeg[#1]{\mbox {$\delta_{_{#1}}^{^+}$}}
\def\maxoutdeg[#1]{\mbox {$\Delta_{_{#1}}^{^+}$}}
\def\indeg[#1,#2]{\mbox {$d_{_{#1}}^{^-}(#2)$}}
\def\minindeg[#1]{\mbox {$\delta_{_{#1}}^{^-}$}}
\def\maxindeg[#1]{\mbox {$\Delta_{_{#1}}^{^-}$}}

\def\dre[#1,#2,#3]{\mbox {${\cal E}^{^{#3}}(#1,#2)$}}
\def\var[#1,#2]{\mbox {${\rm Var}_{_{#1}}(#2)$}}
\def\ls[#1]{\mbox {$\xi^{^{#1}}$}}
\def\hom[#1,#2]{\mbox {${\rm Hom}({#1},{#2})$}}
\def\onvhom[#1,#2]{\mbox {${\rm Hom^{v}}(#1,#2)$}}
\def\onehom[#1,#2]{\mbox {${\rm Hom^{e}}(#1,#2)$}}
\def\core[#1]{\mbox {$#1^{^{\bullet}}$}}
\def\cay[#1,#2]{\mbox {${\rm Cay}({#1},{#2})$}}
\def\sch[#1,#2,#3]{\mbox {${\rm Sch}({#1},{#2},{#3})$}}
\def\cays[#1,#2]{\mbox {${\rm Cay_{s}}({#1},{#2})$}}
\def\dirc[#1]{\mbox {$\stackrel{\rightarrow}{C}_{_{#1}}$}}
\def\cycl[#1]{\mbox {${\bf Z}_{_{#1}}$}}

\begin{document}
\footnotetext[1]{The research of Hossein Hajiabolhassan is supported by ERC advanced grant GRACOL.}
\begin{center}
{\Large \bf  Chromatic Number Via Tur\'an Number}\\
\vspace{0.3 cm}
{\bf Meysam Alishahi$^\dag$ and Hossein Hajiabolhassan$^\ast$\\
{\it $^\dag$ School of Mathematical Sciences}\\
{\it University of Shahrood, Shahrood, Iran}\\
{\tt meysam\_alishahi@shahroodut.ac.ir}\\
{\it $^\ast$ Department of Applied Mathematics and Computer Science}\\
{\it Technical University of Denmark}\\
{\it DK-{\rm 2800} Lyngby, Denmark}\\
{\it $^\ast$ Department of Mathematical Sciences}\\
{\it Shahid Beheshti University, G.C.}\\
{\it P.O. Box {\rm 19839-69411}, Tehran, Iran}\\
{\tt hhaji@sbu.ac.ir}\\
}
\end{center}
\begin{abstract}
\noindent
A Kneser representation ${\rm KG}({\cal H})$ for a graph $G$ is a
bijective assignment of hyperedges of a hypergraph ${\cal H}$ to the vertices of $G$ such that
two vertices of $G$ are adjacent if and only if the corresponding hyperedges are disjoint.
In this paper, we introduce a colored version of the Tur\'an number and use that to
determine the chromatic number of some families of graphs in terms of the generalized Tur\'an number of graphs.
In particular, we determine the chromatic number of every  Kneser multigraph
${\rm KG}({\cal H})$, where the vertex set of ${\cal H}$ is the edge set of
a multigraph $G$ such that the multiplicity of each edge is greater than $1$ and a
hyperedge in ${\cal H}$ corresponds to
a subgraph of $G$
isomorphic to some graph in a fixed prescribed family of simple graphs.\\

\end{abstract}
\section{Introduction}
In this paper, we investigate the chromatic number of graphs. It is a known fact that any graph $G$ has several
Kneser representations. A  {\it Kneser representation} ${\rm KG}({\cal H})$ for a graph $G$ is a bijective assignment of
hyperedges of a hypergraph ${\cal H}$ to the vertices of
$G$ such that two vertices of $G$ are adjacent if and only if the corresponding hyperedges are disjoint. In~\cite{2013arXiv1302.5394A},
in view of Kneser representations of Kneser hypergraphs,
the present authors introduced some lower bounds for their chromatic numbers.
In this regard, for a graph $G$ and for any $1\leq i\leq \chi(G)$, the {\it $i^{th}$ altermatic number }of $G$ was defined as a lower bound for the chromatic number of $G$. Also it was shown that
the $i^{th}$ altermatic number
provides a tight lower bound for the chromatic number. In~\cite{2013arXiv1306.1112M}, It was shown that it is a hard problem to
compute the altermatic number of hypergraphs. Although, we show that
one can evaluate the chromatic number of some families of graphs via their altermatic number.
A graph has various Kneser representations and a hard task to find a suitable lower bound for the altermatic number of a graph is to consider an appropriate representation for that graph.
If we consider a Kneser representation ${\rm KG}({\cal H})$ for a graph,
then one can present some lower and upper bounds for the chromatic number of this graph in terms of the number of vertices of ${\cal H}$ and the independence number of ${\cal H}$.
It is known that the {\it covering number}, i.e., the minimum number of vertices of ${\cal H}$ which meet each
hyperedge of ${\cal H}$, is an upper bound for the chromatic number of $G$. Several interesting results or conjectures related
to the chromatic number of hypergraphs can be reformulated in terms of covering number or  generalized Tur\'an number.
In this paper, we introduce a colored version of the Tur\'an number and use that to
present a lower bound for the chromatic number of graphs.  Moreover,
we determine the chromatic number of some families of graphs in terms of the generalized Tur\'an number.
In this regard, we determine the chromatic number of
some families of path graphs and Kneser multigraphs.\\

This paper is organized as follows.
In the first section, we set up notations and terminologies. In particular, we define the alternating generalized Tur\'an number as a
generalization of the generalized Tur\'an number which provides a lower bound for chromatic number of graphs.
Also, we introduce several
Kneser representations for some well-known families of graphs. In the second section, first we introduce
some lower and upper bounds for chromatic number in terms of the generalized Tur\'an number.
Also, we show that there is a tight relationship between the charomatic number and the generalized Tur\'an number.
Next, we determine the exact value of the chromatic number of every  Kneser multigraph
${\rm KG}({\cal H})$, where the vertex set of the hypergraph ${\cal H}$ is the edge set of
a multigraph $G$ where the multiplicity of each edge is greater than $1$
and hyperedges in ${\cal H}$ correspond to
all subgraphs of $G$ each
isomorphic to some fixed prescribed simple graphs. Moreover, we evaluate the chromatic number of a family of path graphs. In particular, we shall see that
the chromatic number of path graphs
lies between the lower bound and upper bound given in terms of the generalized Tur\'an number.
\subsection{Notations}
First, in this section, we setup some notation and terminology.
Hereafter, the symbol $[n]$ stands for the set $\{1,\ldots, n\}$.
A ({\it multi})~{\it hypergraph} ${\cal H}$ is an ordered pair $(V({\cal H}),E({\cal H}))$,
where $V({\cal H})$ is a finite set, called the set of {\it vertices} of ${\cal H}$, and
$E({\cal H})$ is a family of nonempty
subsets of $V({\cal H})$, called the set of {\it hyperedges} of ${\cal H}$.
The {\it multiplicity} of a hyperedge $e$ is the number of
multiple hyperedges which contain the same vertices as $e$.
Let $N=(N_1,N_2,\ldots,N_r)$, where $N_i$'s are pairwise disjoint subsets of $V({\cal H})$.
The {\it induced hypergraph} ${\cal H}_{|_N}$ has $\displaystyle\cup_{i=1}^rN_i$ and
$\{A\in E({\cal H}):\ \exists i;\ 1\leq i\leq r, A\subseteq N_i\}$ as
vertex set and hyperedge set, respectively.
A {\it hypergraph homomorphism} from ${\cal H}$ to a hypergraph ${\cal F}$ is a map from
the vertex set of ${\cal H}$ to that of ${\cal F}$ such that the image of any hyperedge
of ${\cal H}$ contains some hyperedges of ${\cal F}$.
A {\it $t$-coloring} of ${\cal H}$ is a mapping
$h:V({\cal H})\longrightarrow [t]=\{1,2,\ldots,t\}$ with no monochromatic hyperedge.
Also, the {\it chromatic number}  $\chi({\cal H})$ of ${\cal H}$ is the least positive integer $t$ (the number of colors) such that
there exists a $t$-coloring for ${\cal H}$.
If ${\cal H}$ has some hyperedge of size $1$, then we define its chromatic number to be infinite.
For any hypergraph ${\cal H}=(V({\cal H}),E({\cal H}))$ and positive integer $r\geq 2$,
the {\it general Kneser hypergraph}
${\rm KG}^r({\cal H})$  is an $r$-uniform hypergraph whose vertex set is $E({\cal H})$ and whose hyperedge
set consists of all $r$-tuples of pairwise disjoint hyperedges of ${\cal H}$.
For simplicity of notation, when  $r=2$, the Kneser graph ${\rm KG}^2({\cal H})$
is shown by ${\rm KG}({\cal H})$.

A subset $S \subseteq [n]$ is called {\it $s$-stable}
if any two distinct elements of $S$ are at least
``at distance $s$ apart'' on the $n$-cycle,
that is, $s\leq |i-j|\leq n-s$
for distinct $i,j\in S$.
For a subset $A\subseteq [n]$,
the symbols ${A\choose k}$ and ${A\choose k}_s$ stand
for the set of all $k$-subsets of $A$ and the set of all $s$-stable $k$-subsets of $A$, respectively.
For ${\cal H}_1=([n],{[n]\choose k})$ and ${\cal H}_2=([n],{[n]\choose k}_2)$, two graphs
${\rm KG}({\cal H}_1)$ and  ${\rm KG}({\cal H}_2)$ are termed and denoted by
the {\it  Kneser graph} ${\rm KG}(n,k)$ and the {\it Schrijver graph} ${\rm SG}(n,k)$, respectively.
Also, the generalized Kneser graph ${\rm KG}(n,k,t)$ has ${[n]\choose k}$ as vertex set and two
vertices are adjacent if the size of intersection of corresponding sets is at most $t$. Furthermore, the notations
$C_n, P_n, K_n, K_{m,n}$, and $rK_2$ stand for the {\it $n$-cycle}, the {\it path} with length $n$, i.e., $n+1$ vertices, the {\it complete graph} with $n$ vertices, the complete bipartite graph, and the {\it matching} of size $r$, respectively.
The {\it circular complete graph} $K_{n\over d}$ has $[n]$ as  vertex set and two vertices $i$ and $j$
are adjacent if $d\leq |i-j| \leq n-d$. Circular complete graphs can be considered as a generalization of
complete graphs and they have been studied in the literature,
see~\cite{MR1815614}.

Let $m, n,$ and $r$ be positive integers, where $r\leq m, n$.
For an $r$-subset $A\subseteq [m]$ and an injective map $f:A\longrightarrow [n]$, the ordered pair
$(A,f)$ is said to be an $r$-{\it partial permutation} \cite{MR2892478}.
Let $V_r(m, n)$ denote the set of all $r$-partial permutations.
Two partial permutations $(A, f)$ and $(B, g)$ are said to be {\it intersecting}, if there exists an
$x\in A\cap B$ such that $f(x) = g(x)$.
Note that $V_n(n, n)$ is the set of all $n$-permutations. The {\it permutation graph} $S_r(m, n)$
has $V_r(m, n)$ as vertex set and two $r$-partial permutations are adjacent
if and only if they are~not intersecting.
The structure of maximum independent sets of $S_r(m,n)$ was studied in
several papers, see \cite{MR2009400, MR2489272,MR2202076}.
\subsection{Kneser Representation}
A {\it Kneser representation} for a graph $G$ is an assignment of subsets of a ground set to the vertices of $G$ such that it
assigns distinct subsets to the vertices of $G$ and it satisfies {\it disjoint property}, i.e.,
two vertices are adjacent if and only if the corresponding sets are disjoint.
For more about Kneser representation see~\cite{MR2519945, 2012arXiv1210.6965J}.

For any hypergraph ${\cal H}$ and a family ${\cal L}$ of hypergraphs,
${{\cal H}\choose {\cal L}}$ is a hypergraph whose vertex set is $E({\cal H})$ and whose hyperedge set consists of the hyperedge set of any
subhypergraph of ${\cal H}$ isomorphic to a member of ${\cal L}$.
Hereafter, by abuse of notation, we show the general Kneser hypergraph $\displaystyle{\rm KG}^r\left({{\cal H}\choose {\cal L}}\right)$
by ${\rm KG}^r({\cal H}, {\cal L})$.  Furthermore, for ${\cal L}=\{{\cal F}\}$, where ${\cal F}$ is a hypergraph,
we write ${\rm KG}({\cal H}, {\cal F})$ instead of ${\rm KG}({\cal H}, \{{\cal F}\})$. Also, for simplicity of notation, when  $r=2$, the graph ${\rm KG}^2({\cal H}, {\cal L})$ is shown by
${\rm KG}({\cal H}, {\cal L})$.

Let ${\cal F}=(V({\cal F}), E({\cal F}))$ be a
subhypergraph of ${\cal H}$.
Clearly, the set $V({\cal F})\setminus {\displaystyle\cup_{T\in E(F)}}T$ consists of all {\it isolated vertices} in ${\cal F}$.  If we set
$V'={\displaystyle\cup_{T\in E({\cal F})}}T$, $E'=E({\cal F})$, and ${\cal F}'=(V', E')$, then one can see that
the general Kneser hypergraphs ${\rm KG}^r({\cal H}, {\cal F})$ and  ${\rm KG}^r({\cal H}, {\cal F}')$ are isomorphic.
Hence, for any general Kneser graph ${\rm KG}({\cal H}, {\cal L})$, we may assume that any ${\cal F}\in {\cal L}$ has no isolated vertex. In contrast, we allow the hypergraph ${\cal H}$ to
have some isolated vertices. In fact, we show that
the isolated vertices of a hypergraph ${\cal H}$ may help to present an appropriate lower bound for the chromatic number of the general Kneser hypergraph ${\rm KG}({\cal H}, {\cal L})$.

Here, we introduce some Kneser
representations for some families of well-know graphs:

\begin{enumerate}
\item The Kneser graph ${\rm KG}(n,k)$ ($n\geq 2k$) is isomorphic to ${\rm KG}(nK_2, kK_2)$.

\item The Schrijver graph ${\rm SG}(n,k)$ ($n\geq 2k$) is isomorphic to ${\rm KG}(C_n, kK_2)$.

\item The circular complete graph $K_{n\over d}$ ($n\geq 2d$) is isomorphic to ${\rm KG}(C_n, P_d)$.

\item The generalized Kneser graph  ${\rm KG}(n,k,s)$ ($n\geq k> s$) is isomorphic to
${\rm KG}(K_n^{s+1},K_k^{s+1})$,
where the complete hypergraph $K_n^s$ consists of all $s$-subsets of $[n]$.

\item The permutation graph  $S_r(m,n)$ ($m,n\geq r$) is isomorphic to  ${\rm KG}(K_{m,n}, rK_2)$.
\end{enumerate}
\subsection{Generalized Tur\'an Number}
Throughout this section, let ${\cal H}$ be a finite (multi) hypergraph
and ${\cal F}$ be a family of (multi)
hypergraphs. A hypergraph is called ${\cal F}$-free, if it
has no member of ${\cal F}$ as a subhypergraph.  The maximum number of
hyperedges of an ${\cal F}$-free spanning subhypergraph of ${\cal H}$ (a subhypergraph of ${\cal H}$ with the same vertices as
${\cal H}$)
is denoted by ${\rm ex}({\cal H}, {\cal F})$.
An ${\cal F}$-free spanning subhypergraph of ${\cal H}$ with exactly ${\rm ex}({\cal H}, {\cal F})$ hyperedges is called ${\cal F}$-extremal.
We denote the family of all ${\cal F}$-extremal subhypergraphs of ${\cal H}$ with ${\rm EX}({\cal H}, {\cal F})$.
It is known that ${\rm ex}(K_n, K_3)=\lfloor {n^2 \over 4}\rfloor$. It is usually a hard problem to determine the exact value of ${\rm ex}({\cal H}, {\cal F})$.

Let $\sigma=(e_1, e_2, \ldots, e_t)$ be an ordering of the hyperedges of ${\cal H}$, where $t=|E({\cal H})|$. An alternating $2$-coloring of $E({\cal H})$ of length $l$ with respect to the ordering $\sigma$, or simply, an alternating $2$-coloring of length $l$, is an assignment of two colors blue and red to exactly $l$ hyperedges of ${\cal H}$ such that any two consecutive colored hyperedges (with respect to the ordering $\sigma$) have different colors. For instance, let $\sigma=(e_1, e_2, e_3, e_4, e_5, e_6)$ and $\sigma'=(e_2, e_1, e_3,e_4,  e_5, e_6)$ be two orderings of the edges of the complete graph $K_4$. Also, let $G$ be  a subgraph of $K_4$, where $E(G)=\{e_1,e_2, e_4,e_6\}$. One can check that the coloring $f:\ E(G)\rightarrow \{B,R\}$, where $f(e_1)=R, f(e_2)=B, f(e_4)=R$, and $f(e_6)=B$,
is an alternating $2$-coloring
of $E(K_4)$ of  length $4$ with respect to the ordering $\sigma$. But this coloring is~not an alternating $2$-coloring of $E(K_4)$  of  length $4$ with respect to the ordering $\sigma'$.

For a given $2$-coloring of a subset of hyperedges of ${\cal H}$, a hyperedge with color red (resp. blue) is called a red (resp. blue) hyperedge. Moreover, the {\it red subhypergraph} ${\cal H}^R$ (resp. {\it blue subhypergraph} ${\cal H}^B$) of $H$ is
a spanning subhypergraph of ${\cal H}$
whose hyperedge set
consists of all red (resp. blue)  hyperedges. For instance, in the above-mentioned
example, for the alternating coloring with respect to the ordering $\sigma$, we have $E({\cal H}^R)=\{e_1,e_4\}$ and
$E({\cal H}^B)=\{e_2,e_6\}$.

For an ordering $\sigma$ of $E({\cal H})$,
the maximum possible length of an alternating $2$-coloring of $E({\cal H})$ with respect to the ordering
$\sigma$ such that both of the corresponding red and  blue subhypergraphs
(resp. at least one of the red subhypergraph and the blue subhypergraph) are
 ${\cal F}$-free is denoted by
${\rm ex}_{alt}({\cal H}, {\cal F},\sigma)$ (resp. ${\rm ex}_{salt}({\cal H}, {\cal F},\sigma)$).

Set
$${\rm ex}_{alt}({\cal H}, {\cal F})=\min\{{\rm ex}_{alt}({\cal H}, {\cal F},\sigma):\sigma\ is\ an\ ordering\ of\ E({\cal H})\}.$$
$${\rm ex}_{salt}({\cal H}, {\cal F})=\min\{{\rm ex}_{salt}({\cal H}, {\cal F},\sigma):\sigma\ is\ an\ ordering\ of\ E({\cal H})\}.$$
Note that if we assign alternatively two colors red and blue to the hyperedge set of a member of ${\rm EX}({\cal H}, {\cal F})$
with respect to an arbitrary ordering
$\sigma$, then one can conclude that ${\rm ex}({\cal H}, {\cal F}) \leq {\rm ex}_{alt}({\cal H}, {\cal F},\sigma)$.
Here we present an example to show equality holds.
Consider the complete graph $K_4$ and let $E(K_4)=\{e_1,e_2,e_3,e_4,e_5,e_6\}$
such that $\{e_{2i-1},e_{2i}\}$  forms a
matching for any $1\leq i \leq 3$. One can see that for the path $P_2$, we have  ${\rm ex}(K_4, P_2)=2$.
Also, one can check that for the ordering
$\sigma=(e_1, e_2, e_3, e_4, e_5, e_6)$, we have ${\rm ex}_{alt}(K_4, P_2,\sigma)=2$.
Consequently, ${\rm ex}(K_4, P_2)={\rm ex}_{alt}(K_4, P_2)=2$. Also, it is straightforward to
see that
${\rm ex}_{alt}({\cal H}, {\cal F},\sigma) \leq 2  {\rm ex}({\cal H}, {\cal F})$. Accordingly,
$${\rm ex}({\cal H}, {\cal F})\leq {\rm ex}_{alt}({\cal H}, {\cal F}) \leq 2  {\rm ex}({\cal H}, {\cal F}).$$
Similarly, one can see that ${\rm ex}({\cal H}, {\cal F})+1\leq {\rm ex}_{salt}({\cal H}, {\cal F}) \leq 2  {\rm ex}({\cal H}, {\cal F})+1.$
\subsection{Altermatic Number}
Let  $x_1,x_2,\ldots,x_n$ be a sequence of $\{-1,+1\}$.
The subsequence $x_{j_1}, x_{j_2},\ldots,x_{j_m}$ (${j_1}<{j_2}<\cdots<{j_m}$)
is said to be an {\it alternating subsequence} if any two consecutive terms in this subsequence are distinct, i.e., $x_{j_i}x_{j_{i+1}}<0$, for each $i\in\{1,2,\ldots,m-1\}$. For an $X=(x_1,x_2,\ldots,x_n)\in\{-1,0,+1\}^n\setminus \{(0,0,\ldots,0)\}$,
 the length of a longest alternating subsequence of nonzero terms in $X$ is denoted by $alt(X)$.
For instance, if $X=(+1,-1,0, -1,0,+1,+1,-1)$, then $alt(X)=4$.

One can consider the set of vectors of $\{-1,0,+1\}^n\setminus \{(0,0,\ldots,0)\}$ as the set of all
{\it signed subsets} of $[n]$, that is, the
family of all $(X^+,X^-)$ of disjoint subsets of $[n]$. Precisely,
for any vector $X=(x_1,x_2,\ldots,x_n)\in\{-1,0,+1\}^n\setminus \{(0,0,\ldots,0)\}$, we define
$$X^+=\{i\in [n]:\ x_i=+1\}, \quad X^-=\{i\in [n]:\ x_i=-1\}.$$
Note that a vector $X$  determines uniquely $X^+$ and $X^-$ and vise versa. Therefore, by abuse of notation, we can set $X=(X^+,X^-)$. Throughout this paper, for any $X=(x_1,x_2,\ldots,x_n)\in\{-1,0,+1\}^n\setminus \{(0,0,\ldots,0)\}$,
we use these representations interchangeably, i.e., $X=(x_1,x_2,\ldots,x_n)$ or $X=(X^+,X^-)$.
For a ground set $V=\{v_1,v_2,\ldots,v_n\}$, denote the set of \emph{linear orderings} of $V$ by $L(V)$ and let  $\sigma: v_{i_1}<v_{i_2}<\cdots<v_{i_n}$ be a linear ordering of $V$.
For any $X\in\{-1,0,+1\}^n\setminus\{(0,0,\ldots,0)\}$, define
$$X^+_\sigma=\{v_{i_j}:\ j\in[n]\ \&\  x_j=+1\}, \quad  X^-_\sigma=\{v_{i_j}:\ j\in[n]\ \&\  x_j=-1\},$$
and $X_\sigma=(X^+_\sigma,X^-_\sigma)$.

Let ${\cal H}=(V,E)$ be a hypergraph and $\sigma\in L_V$  be a linear ordering. For any positive integer $i$, set $alt_{\sigma}({\cal H},i)$ to be the largest integer $k$ such that there
exists an $X\in\{-1,0,+1\}^n\setminus \{(0,0,\ldots,0)\}$
with $alt(X)=k$  and that the hypergraph ${\cal H}_{|X_\sigma}$ has chromatic number at most $i-1$.
One can see that $alt_{\sigma}({\cal H},1)$ is the largest integer $k$ such that there
exists an $X\in\{-1,0,+1\}^n\setminus \{(0,0,\ldots,0)\}$
with $alt(X)=k$  and that none  of  $X^+_\sigma$ and  $X^-_\sigma$
 contains any hyperedge of ${\cal H}$.
If for each $X\in\{-1,0,+1\}^n\setminus \{(0,0,\ldots,0)\}$ either $X^+_\sigma$ or $X^-_\sigma$ has
some hyperedges of ${\cal H}$, i.e., every singleton is a hyperedge of ${\cal H}$,  then we define $alt_{\sigma}({\cal H},1)=0$.
Set $alt({\cal H},i)=\min\{alt_\sigma({\cal H},i):\ \sigma\in L_V\}$.
Now we are ready to define the {\it $i^{\rm th}$ altermatic number}  of a graph $G$ as follows.
$$
\zeta(G,i)=\displaystyle\max_{{\cal H}} \{|V({\cal H})|-alt({\cal H},i)+i-1: {\rm KG}({\cal H})\longleftrightarrow G\}.
$$
It should be mentioned that the $i^{th}$ alternation number of graphs was defined in \cite{2013arXiv1302.5394A} in a different way.
In what follows, we show that these definitions are equivalent.
In {\rm \cite{2013arXiv1302.5394A}}, the present authors defined the {\it $i^{\rm th}$ alternation number} for graphs as follows.
For a hypergraph ${\cal H}=([n],E)$, a vector $X\in\{-1,0,+1\}^n\setminus\{(0,0,\ldots,0)\}$, and a linear ordering $\sigma: i_1<i_2<\cdots<i_n$ of $[n]$,
set $alt_\sigma(X)=alt(x_{i_1},x_{i_2},\ldots,x_{i_n})$. Note that for $I:1<2<\cdots<n$, we have
$X^+_I=X^+$ and $X^-_I=X^-$. Moreover, for any positive integer $i$,
set $alt'_{\sigma}({\cal H},i)$ to be the largest integer $k$ such that there
exists an $X=(x_1,x_2,\ldots,x_n)\in\{-1,0,+1\}^n\setminus \{(0,0,\ldots,0)\}$
with $alt((x_{i_1},x_{i_2},\ldots,x_{i_n}))=k$  and that the chromatic number of
hypergraph ${\rm KG}({\cal H}_{|X})$ is at most $i-1$.
Now define $alt'({\cal H},i)=\min\{alt'_\sigma({\cal H},i):\ \sigma\in L_{[n]}\}$.
\begin{lem}
For any hypergraph ${\cal H}=([n],E)$ and positive integer $i$,
we have $alt({\cal H},i)=alt'({\cal H},i)$.
\end{lem}
\begin{proof}{
Consider an arbitrary ordering $\sigma: i_1<i_2<\cdots<i_n \in L_{[n]}$.
Let $alt_\sigma(F,i)=k$.
In view of the definition of $alt_\sigma(F,i)$, there is an $X=(x_1,x_2,\ldots,x_n)\in\{+1,0,-1\}^n\setminus\{(0,0,\ldots,0)\}$
such that $alt(X)=k$ and the chromatic number of
the hypergraph ${\rm KG}(F_{|X_\sigma})$ is at most $i-1$.
Define $\beta\in S_n$ such that $\beta(j)=i_j$ for $j=1,2,\ldots,n$.
Now consider $\gamma:\beta^{-1}(1)<\beta^{-1}(2)<\cdots<\beta^{-1}(n)\in L_{[n]}$ and
let
$$Y=(y_1,y_2,\ldots,y_n)=(x_{\beta^{-1}(1)},x_{\beta^{-1}(2)},\ldots,x_{\beta^{-1}(n)}).$$
Note that $(y_{i_1},y_{i_2},\ldots,y_{i_n})=X$ and so
$alt((y_{i_1},y_{i_2},\ldots,y_{i_n}))=alt(X)=k.$
Also, we have
$$Y^+=\{i\in[n]:\ y_i=+1\}=\{i\in[n]:\ x_{\beta^{-1}(i)}=+1\}=\{\beta(j)\in[n]:\ x_j=+1\}=X^+_\sigma$$
and similarly, $Y^-=X^-_\sigma$. Accordingly, we have ${\rm KG}(F_{|Y})={\rm KG}(F_{|X_\sigma})$. Therefore,
the chromatic number of
the hypergraph ${\rm KG}(F_{|Y})$ is at most $i-1$.
It implies $alt'_\gamma(F,i)\geq alt_\sigma(F,i)$; and consequently, since
$\sigma$ is an arbitrary ordering, we have $alt'(F,i)\geq alt(F,i)$.

Now, let $\gamma: i_1<i_2<\cdots<i_n \in L_{[n]}$ be an arbitrary ordering of $[n]$.
Assume that $alt'_\gamma(F,i)=k'$.
In view of the definition of $alt'_\gamma(F,i)$, there exists an $X=(x_1,x_2,\ldots,x_n)\in\{+1,0,-1\}^n\setminus\{(0,0,\ldots,0)\}$
such that $alt((x_{i_1},x_{1_2},\ldots,x_{i_n}))=k'$ and the chromatic number of
the hypergraph ${\rm KG}(F_{|X})$ is at most $i-1$.
Now, let $Z=(z_1,z_2,\ldots,z_n)=(x_{i_1},x_{i_2},\ldots,x_{i_n})$.
Note that $alt(Z)=k'$,
$$Z^+_\gamma=\{i_j:\ z_j=+1\}=\{i_j:\ x_{i_j}=+1\}=X^+,$$
and similarly $Z^-_\sigma=X^-$. Hence,
we have $${\rm KG}(F_{|X})={\rm KG}(F_{|Z_\gamma}).$$
Therefore, the chromatic number of
the hypergraph ${\rm KG}(F_{|Y_\gamma})$ is at most $i-1$.
It implies that $alt_\gamma(F,i)\geq alt'_\gamma(F,i)$.
Since $\gamma$ is an arbitrary ordering, we have $alt(F,i)\geq alt'(F,i)$
which completes the proof.
}\end{proof}

Now, we are in a position to introduce a lower bound for chromatic number of any graph $G$
in terms of its altermatic number. The next theorem was expressed
in terms of  $alt'(F,i)=alt(F,i)$ in \cite{2013arXiv1302.5394A}.
\begin{alphthm}\label{generalsalt}{\rm \cite{2013arXiv1302.5394A}}
For any graph $G$ and any positive integer $i\leq \chi(G)+1$, we have
$$\chi(G) \geq \zeta(G,i).$$
\end{alphthm}

In view of  simplicity, we define two parameters which help us to determine the chromatic number of some families of graphs.
For a hypergraph ${\cal H}=(V,E)$ and a linear ordering $\sigma:v_1<v_2<\cdots<v_n\in L_V$,
we set $alt_{\sigma}({\cal H})$ (resp. $salt_{\sigma}({\cal H})$) to be the largest integer $k$ such that there
exists an $X\in\{-1,0,+1\}^n\setminus \{(0,0,\ldots,0)\}$
with $alt(X)=k$  and that none (resp. at most one) of  $X^+_\sigma$ and  $X^-_\sigma$
contains any (resp. some) hyperedge of ${\cal H}$. Note that $alt_{\sigma}({\cal H})=alt_{\sigma}({\cal H},1)$ and
$alt_{\sigma}({\cal H},2) \leq salt_{\sigma}({\cal H})$. Also,
$alt_{\sigma}({\cal H}) \leq alt_{\sigma}({\cal H},2) \leq salt_{\sigma}({\cal H}) $ and the equality can hold. For instance, one can see that for $k\geq 2$ and $I:1<2<\cdots<n$,
$alt_{I}({[n]\choose k}_{2})=salt_{I}({[n]\choose k}_{2})=2(k-1)+1$.
Now, set $alt({\cal H})=\min\{alt_{\sigma}({\cal H});\ \sigma\in L_{V({\cal H})}\}$ and $salt({\cal H})=\min\{salt_{\sigma}({\cal H});\ \sigma\in L_{V({\cal H})}\}$.
For a graph $G$, we define
{\it altermatic number} $\zeta(G)$ and {\it strong altermatic number} $\zeta_s(G)$,
respectively, as follows
$$
\begin{array}{ccl}
\zeta(G) & = &  \displaystyle \max_{{\cal H}}\{|V({\cal H})|-alt({\cal H}): {\rm KG}({\cal H})\longleftrightarrow G\}, \\ \\
\zeta_s(G) & = & \displaystyle \max_{{\cal H}}\{|V({\cal H})|+1-salt({\cal H}): {\rm KG}({\cal H})\longleftrightarrow G\}.
\end{array}
$$

Note that $\zeta(G)=\zeta(G,1)$   and $\zeta_s(G)\leq  \zeta(G,2)$; consequently, in view of Theorem~\ref{generalsalt}, we have
the following lower bound for the chromatic number of graphs.

\begin{alphthm}\label{salt}{\rm \cite{2013arXiv1302.5394A}}
For any graph $G$, we have
$$\chi(G) \geq \max\{ \zeta(G), \zeta_s(G)\}.$$
\end{alphthm}
We can usually find an appropriate upper bound for $alt_{\sigma}(F,2)$ by computing $salt_{\sigma}(F)$.


In the rest of this paper, we determine the chromatic number of some families of graphs by applying Theorem~\ref{salt}, although,
Theorem~\ref{generalsalt} is stronger than Theorem~\ref{salt}.

In view of Theorem~\ref{salt}, we can consider altermatic number and strong altermatic number as tight lower bounds for chromatic number of graphs. For instance, in view of representation of the Kneser graph ${\rm KG}(n,k)$ and the Schrijver graph ${\rm SG}(n,k)$, one can see that
$\min\{\zeta({\rm KG}(n,k)), \zeta_s({\rm KG}(n,k))\}=n-2k+2$,
$\zeta({\rm SG}(n,k))\geq n-2k+1$, and $\zeta_s({\rm SG}(n,k))= n-2k+2$.
Note that a graph has several Kneser representations and different
Kneser representations can lead us to different lower bounds for chromatic number.
For instance, consider the five cycle. Set ${\cal F}=(V({\cal F}),E({\cal F}))$
and ${\cal H}=(V({\cal H}),E({\cal H}))$ as follows.
$$V({\cal F})=\{1,2,3,4,5\}\ \& \ E({\cal F})=\{\{1,2\}, \{2,3\}, \{3,4\}, \{4,5\}, \{1,5\}\},$$
$$V({\cal H})=\{1,2,3,4,5,a,b,c,d,e\}\ $$
and
$$ \ E({\cal H})=\displaystyle\left\{\{1,2\}, \{2,3\}, \{3,4\}, \{4,5\}, \{1,5\}\right\}.$$
Note that the hypergraphs ${\cal F}$ and ${\cal H}$ have the same hyperedge set and the
hypergraph ${\cal H}$ has $5$ isolated vertices. One can check that
$alt({\cal F})=3$ and this shows that the chromatic number of five cycle is at least two. Although, one can check that
$alt({\cal H})=7$ and this leads us to  $3$ as a lower bound for the chromatic number. To see this, consider the ordering
$$\sigma=(1, a, 2, b, 3, c, 4, d, 5, e).$$
For a contradiction, suppose that $alt_\sigma({\cal H})\geq 8$. This means that there exists an
$X=(x_1,x_2,\ldots,x_{10})\in \{-1,0,+1\}^{10}\setminus \{(0,0,\ldots,0)\}$ such that
$alt(X)= 8$ and that both of $X_{\sigma}^+$ and  $X_{\sigma}^-$
contain no hyperedge of ${\cal H}$. We may assume that $X$ has exactly $8$ nonzero coordinates. Otherwise,
suppose that $x_{i_1},x_{i_2},\ldots,x_{i_8}$ is an alternating sequence of nonzero terms in $X$, where
$i_1<i_2<\cdots<i_8$.
By changing the value of $x_j$ to $0$ for all $j\not \in\{i_1,i_2,\ldots,i_8\}$, one can obtain a
$Y\in \{-1,0,+1\}^{10}\setminus \{(0,0,\ldots,0)\}$
such that $alt(Y)=alt(X)=8$ and that both of $Y_{\sigma}^+$ and  $Y_{\sigma}^-$
contain no hyperedge of ${\cal H}$.
Therefore,
we may suppose that $alt(X)=8$ and $X$ has exactly $8$ nonzero coordinates.
Also, for any odd integer $1\leq i\leq 9$, at least one of $x_i$, $x_{i+1}$, or $x_{i+2}$ is zero; since otherwise,
one can conclude that either $X_{\sigma}^+$ or  $X_{\sigma}^-$ contains a hyperedge of ${\cal H}$ (indices are considered in $\mathbb{Z}_9=\{1,2,\ldots,9\}$).
Now, by a double counting, one can see that there are at least $3$ zero coordinates in $X$ which implies $alt(X)\leq 7$, a contradiction.
\section{General Kneser Graphs}
\subsection{Lower and Upper Bound}
In this section, we introduce some tight lower and upper bounds for the
chromatic number of general Kneser hypergraphs.
In fact, by presenting an upper bound for alternating
Tur\'an number (resp. strong  alternating
Tur\'an number), we obtain a lower bound for chromatic number. Next, in view of these bounds, we determine
the chromatic number of some families of graphs.
\begin{lem}\label{lowupa}
Let ${\cal H}$ be a hypergraph. For any positive integer $r\geq 2$,
 $${|V({\cal H})|-r.\alpha({\cal H})\over r-1}\leq
 \chi({\rm KG}^r({\cal H})) \leq \left\lceil{|V({\cal H})|-\alpha({\cal H})\over r-1}\right\rceil.$$
\end{lem}
\begin{proof}{
First, we present a hypergraph homomorphism
$$\phi: {\rm KG}^r(|V({\cal H})|, \alpha({\cal H})+1)\longrightarrow {\rm KG}^r({\cal H}),$$
to show that the lower bound holds.
Let $V({\cal H})=\{v_1,\ldots,v_{n}\}$, where $n=|V({\cal H})|$. For any vertex
$A=\{i_1,\ldots,i_l\}$ of ${\rm KG}^r(|V({\cal H})|, \alpha({\cal H})+1)$, where $l=\alpha({\cal H})+1$, set
$\phi(A)$ to be an arbitrary hyperedge of ${\cal H}$ which is a subset of
$\{v_{i_1},\ldots,v_{i_l}\}$.
It is easy to see that $\phi$ is a hypergraph homomorphism. Consequently, 
$$\chi({\rm KG}^r(|V({\cal H})|, \alpha({\cal H})+1))=\left\lceil{|V({\cal H})|-r.\alpha({\cal H})\over r-1}\right\rceil
\leq \chi({\rm KG}^r({\cal H})).$$ 
Now we prove the upper bound. Let $S\subseteq V({\cal H})$ be an independent set of size $\alpha({\cal H})$.
Also,
let $|V({\cal H})\setminus S|=t=|V({\cal H})|-\alpha({\cal H})$. Consider a partition
$S_1\cup\cdots \cup S_q$ of $V({\cal H})\setminus S$, where
$S_i$'s are pairwise disjoint, $|S_1|=\cdots=|S_{q-1}|=r-1$, and $0<|S_q| \leq r-1$ ($q=\lceil {t\over r-1 }\rceil$).
For any hyperedge of ${\cal H}$, assign the color
$i$ to it, where  $i$ is the smallest positive integer such that this hyperedge
has nonempty intersection with $S_i$. One can check that this assignment provides a
proper coloring for the Kneser hypergraph ${\rm KG}^r({\cal H})$.
}\end{proof}
For a hypergraph ${\cal H}$, a vertex cover of ${\cal H}$ is a subset of $V({\cal H})$
which meets each hyperedge of ${\cal H}$. The minimum size of such a subset is called the vertex covering number of ${\cal H}$ and is denoted by $\beta({\cal H})$. It is well-known that $\beta({\cal H})=|V({\cal H})|-\alpha({\cal H})$.
In the next theorem, we introduce a new formula to determine the chromatic number of graphs in terms of the covering number of some related hypergraphs.
\begin{thm}
For any graph $G$, we have
$$\chi(G)=\min\{\beta({\cal H}):\ {\rm KG}({\cal H})\cong G\}.$$
\end{thm}
\begin{proof}{
Note that by Lemma~\ref{lowupa}, it is enough to show that
$$\chi(G)\geq\min\{\beta({\cal H}):\ {\rm KG}({\cal H})\cong G\}.$$
To this end,
consider a Kneser representation ${\rm KG}(\bar{\cal H})$ of $G$. Also,
let $c:{\rm KG}(\bar{\cal H})\cong G \longrightarrow C $
be a proper coloring of ${\rm KG}(\bar{\cal H})$ with $\chi(G)$ colors such that
$C\cap V(\bar{\cal H})=\varnothing$.
Consider the hypergraph ${\cal H}$, where $V({\cal H})=V(\bar{\cal H})\cup C$
and $E({\cal H})=\{e\cup c(e):\ e\in E(\bar{\cal H})\}.$
One can check that ${\rm KG}({\cal H})\cong G$; and consequently,
$\chi({\rm KG}({\cal H}))=\chi(G)$.
Moreover, $V(\bar{\cal H})\subseteq V({\cal H})$ is an independent set of ${\cal H}$.
Therefore, $\alpha({\cal H})\geq |V(\bar{\cal H})|=|V({\cal H})|-\chi({\rm KG}({\cal H}))$.
In view of the upper bound of the previous lemma, we have
$\chi(G)=\chi({\rm KG}({\cal H}))=|V({\cal H})|-\alpha({\cal H})$ which completes the proof.
}\end{proof}

Now, we introduce some bounds for the chromatic number of general Kneser hypergraphs in terms of the generalized Tur\'an number. The relationship between
the chromatic number of some families of general Kneser graphs and the generalized Tur\'an number has been
studied by several researchers with different notations. Frankl~\cite{MR797510} determined the chromatic
number of generalized Kneser graph ${\rm KG}(n,k,1)$ provided that $n$ is sufficiently large.
Note that ${\rm KG}(n,k,1)$ is isomorphic to ${\rm KG}(K_n,K_k)$. Also,
the chromatic number of ${\rm KG}(K_n,C_m)$ was investigated in \cite{MR3073134} and the
authors obtained independently the same result of Tort~\cite{MR689810} when $m=3$ and $n\geq 10$.
In fact, Tort has shown that
$\chi({\rm KG}(K_n,C_3))={n \choose 2}-{\rm ex}(K_n,C_3)=\lfloor {(n-1)^2 \over 4}\rfloor$ for $n\geq 5$.

\begin{lem}\label{lowup}
Let ${\cal H}$ be a hypergraph and ${\cal F}$ be a family of hypergraphs. For any positive integer $r\geq 2$,
 $${|E({\cal H})|-r.{\rm ex}({\cal H}, {\cal F})\over r-1}\leq
 \chi({\rm KG}^r({\cal H},{\cal F})) \leq \left\lceil{|E({\cal H})|-{\rm ex}({\cal H}, {\cal F})\over r-1}\right\rceil.$$
\end{lem}
\begin{proof}{
One can check that $\alpha({{\cal H} \choose {\cal F}})={\rm ex}({\cal H}, {\cal F})$. Now, in view of Lemma~\ref{lowupa},
the assertion holds.
}\end{proof}
\begin{lem}\label{ex}
For any hypergraph ${\cal H}$ and a family ${\cal F}$ of hypergraphs,
$$\displaystyle\begin{array}{rllll}
|E({\cal H})|-{\rm ex}_{alt}({\cal H}, {\cal F}) &    \leq &  \chi({\rm KG}({\cal H},{\cal F})) & \leq  & |E({\cal H})|-{\rm ex}(H, {\cal F}),\\
& & & &  \\
|E({\cal H})|+1-{\rm ex}_{salt}({\cal H}, {\cal F}) &    \leq &  \chi({\rm KG}(H,{\cal F})) & \leq  & |E({\cal H})|-{\rm ex}({\cal H}, {\cal F}).
\end{array}
$$
\end{lem}
\begin{proof}{
One can check that $alt({{\cal H} \choose {\cal F}})={\rm ex}_{alt}({\cal H}, {\cal F})$ and
$salt({{\cal H} \choose {\cal F}})={\rm ex}_{salt}({\cal H}, {\cal F})$. Now, in view of Theorem~\ref{salt} and Lemma~\ref{lowup},
the assertion holds.
}\end{proof}
In view of Lemma~\ref{ex},  if
${\rm ex}(H, {\cal F})={\rm ex}_{alt}(H, {\cal F})$ or ${\rm ex}(H, {\cal F})+1={\rm ex}_{salt}(H, {\cal F})$, then
$$\chi({\rm KG}(H,{\cal F}))=|E(H)|-{\rm ex}(H, {\cal F}).$$
Here, we present several examples to show that the upper bound mentioned in Lemma~\ref{ex} is sharp.
We showed that ${\rm ex}(K_4, P_2)={\rm ex}_{alt}(K_4, P_2)=2$. Consequently,
$\chi({\rm KG}(K_4, P_2))=|E(K_4)|-{\rm ex}(K_4, P_2)=4$.
Furthermore, it is known that if $n$ is sufficiently large, then the
Tur\'an number of $K_k$, i.e. ${\rm ex}(K_n,K_k)$,  is ${n\choose 2}-(k-1){s \choose 2}-rs$, where $n=(k-1)s+r, 0\leq r< k-1$. Hence, in view of
the result of Frankl~\cite{MR797510}, one can see that there exists an
integer $n_k$ such that for $n\geq n_k$, we have $\chi({\rm KG}(K_n,K_k))={n\choose 2}-{\rm ex}(K_n,K_k)$.
Also, the following example confirms that the lower bound mentioned in Lemma~\ref{lowup} is sharp.
One can see that if $G=C_n$ and ${\cal F}$ consists of all subgraphs of $C_n$ with exactly $k$-edges, then
$\chi({\rm KG}(G, {\cal F}))=\chi(${\rm SG}(n,k)$)=n-2k+2=|E(G)|-2.{\rm ex}(G,{\cal F})$.
\subsection{Kneser Multigraph}
In what follows, by Lemma~\ref{lowup}, we determine the chromatic number of some family of graphs.
In particular, we determine $\chi({\rm KG}(G,{\cal F}))$ in terms of the generalized Tur\'an number whenever $G$ is a multigraph and ${\cal F}$ is a family of simple graphs.

\begin{thm}\label{multi}
Let $G$ be a multigraph such that the multiplicity of each edge of $G$ is at least two.
If ${\cal F}$ is a family of simple subgraphs of $G$, then we have
$$\chi({\rm KG}(G,{\cal F}))= \zeta({\rm KG}(G,{\cal F}))=|E(G)|-{\rm ex}(G, {\cal F}).$$
In particular, if the multiplicity of each edge is an even integer, then
$$\chi({\rm KG}(G,{\cal F}))= \zeta_s({\rm KG}(G,{\cal F})).$$
\end{thm}
\begin{proof}{
First, we show that $\chi({\rm KG}(G,{\cal F}))= \zeta({\rm KG}(G,{\cal F}))$.
In view of Lemma~\ref{ex}, it is sufficient to show that ${\rm ex}_{alt}(G, {\cal F})={\rm ex}(G, {\cal F})$.
Assume that $E(G)=I_1\cup\cdots\cup I_t$ is a partition of $E(G)$, where
for any $1\leq i\leq t$, there are two distinct vertices $u$ and $v$ such that $I_i$ consists of all edges
incident with both of $u$ and $v$. Since the multiplicity of each edge is at least two, we have
for any $1\leq i\leq t$, $|I_i|\geq 2$. Consider an ordering
$\sigma$ for the edge set of $G$ such that all edges of each $I_j$ appear consecutively in
the ordering $\sigma$ (they form an interval in this ordering). Now,
we show that ${\rm ex}_{alt}(G, {\cal F},\sigma)={\rm ex}(G, {\cal F})$.
To see this, consider an alternating coloring of a subset of edges of
$G$ of length more than ${\rm ex}(G, {\cal F})$ with respect to the ordering $\sigma$. In view of the ordering $\sigma$, for any $1\leq j\leq t$ or
both colors are assigned to some edges of $I_j$, or exactly one edge of $I_j$ is colored, or no color
is assigned to edges of $I_j$.
Suppose that $C_1=\{I_{j_1},I_{j_2},\ldots,I_{j_{k_1}}\}$ consists of all
$I_{j_i}$'s such that both colors are assigned to some edges of $I_{j_i}$,
$C_2=\{I_{l_1},I_{l_2},\ldots,I_{l_{k_2}}\}$ consists of all
$I_{l_i}$'s such that just red color is assigned to some edges of $I_{l_i}$,
and $C_3=\{I_{t_1},I_{t_2},\ldots,I_{t_{k_3}}\}$ consists of all
$I_{t_i}$'s such that just blue color is assigned to some edges of $I_{t_j}$. Without loss of generality,
suppose that $|C_2|=k_2\geq k_3=|C_3|$. Denote the spanning subgraph
containing all  the red edges (resp.  the blue edges) by $G^R$ (resp. $G^B$).
We show that the subgraph $G^R$ contains some member of ${\cal F}$.
On the contrary, suppose that the assertion is false.
Therefore, the subgraph $H$ consisting of all edges in
$(\displaystyle\cup_{i=1}^{k_1}I_{j_i})\cup(\displaystyle\cup_{i=1}^{k_2} I_{l_i})$ contains no member of ${\cal F}$
and so $|E(H)|\leq {\rm ex}(G, {\cal F})$.
This implies that
$$\begin{array}{lll}
  |E(G^R)|+|E(G^B)| & \leq & \displaystyle\sum_{i=1}^{k_1}|I_{j_i}|+k_2+k_3\\
                    & \leq & \displaystyle\sum_{i=1}^{k_1}|I_{j_i}|+2k_2 \\
                    & \leq & \displaystyle\sum_{i=1}^{k_1}|I_{j_i}|+\displaystyle \sum_{i=1}^{k_2}|I_{l_i}| \\
                    &   =  &  |E(H)|\leq {\rm ex}(G, {\cal F})
\end{array}$$
which contradicts our assumption that the number of colored edges, i.e., $|E(G^R)|+|E(G^B)|$, is more than
${\rm ex}(G, {\cal F})$.

To prove the second part of theorem,
in view of Lemma~\ref{ex}, it is sufficient to show that ${\rm ex}_{salt}(G, {\cal F})={\rm ex}(G, {\cal F})+1$.
To see this, consider an alternating coloring of a subset of edges of
$G$ of length more than ${\rm ex}(G, {\cal F})+1$ with respect to the aforementioned ordering $\sigma$.
Consider $C_1$, $C_2$, and $C_3$ as defined in the previous part.
We show that each of the red subgraph $G^R$ and the blue subgraph $G^B$ contains some member of ${\cal F}$.
Note that ${\rm ex}(G, {\cal F})$ is an even integer.
Also, both of $G^R$ and $G^B$ have at least $1+{{\rm ex}(G, {\cal F})\over 2}$ edges.
On the contrary, suppose that $G^R$ contains no member of ${\cal F}$.
Therefore, the subgraph $H$ consisting of all edges in
$(\displaystyle\cup_{i=1}^{k_1}I_{j_i})\cup(\displaystyle\cup_{i=1}^{k_2}I_{l_i})$ contains no member of ${\cal F}$
and so $|E(H)|\leq {\rm ex}(G, {\cal F})$.
Since all multiplicities are even, for any $i$, red color can be assigned to at most ${|I_i|\over 2}$ edges of $I_i$.
This implies that
$$1+{{\rm ex}(G, {\cal F})\over 2}\leq|E(G^R)|\leq \displaystyle{1\over 2}\sum_{i=1}^{k_1}|I_{j_i}|+k_2
\leq{|E(H)|\over 2}\leq {{\rm ex}(G, {\cal F})\over 2},$$
which is a contradiction.
Similarly, the subgraph $G^B$ has some member of ${\cal F}$ and this implies
$\chi({\rm KG}(G,{\cal F}))= \zeta_s({\rm KG}(G,{\cal F}))$.
}\end{proof}
\begin{thm}
Let $H$ be a simple graph and ${\cal F}$ be a family of subgraphs of $H$.
Assume that $G$ is obtained from $H$ by giving the same multiplicity $r\geq 2$ to some edges of $H$. If
the subgraph of $H$ corresponding to the edges of $G$ with
multiplicity $r$ has an  ${\cal F}$-free subgraph with ${\rm ex}(H,{\cal F})$ edges, then
$$\chi({\rm KG}(G,{\cal F}))=\zeta({\rm KG}(G,{\cal F}))=|E(G)|-{\rm ex}(G,{\cal F}).$$
In particular, if $r$ is an even integer, then
$$\chi({\rm KG}(G,{\cal F}))= \zeta_s({\rm KG}(G,{\cal F})).$$
\end{thm}
\begin{proof}{
First, note that ${\rm ex}(G, {\cal F})=r.{\rm ex}(H, {\cal F})$; and therefore,
in view of Lemma~\ref{ex}, it is sufficient to show that ${\rm ex}_{alt}(G, {\cal F})=r.{\rm ex}(H, {\cal F})$.
Suppose that $E(G)=I_1\cup\cdots\cup I_t\cup I_{t+1}\cup\cdots\cup I_m$ is a partition of $E(G)$, where
for any $1\leq i\leq m$, there are two distinct vertices $u$ and $v$ such that $I_i$ consists of all edges
incident with both of $u$ and $v$, and moreover,
$|I_i|=r\geq 2$ for any $1\leq i\leq t$; otherwise,  $|I_i|=1$. Consider an ordering
$\sigma$ for the edge set of $G$ such that all edges of each $I_j$ appear consecutively in
the ordering $\sigma$ (they form an interval in this ordering) and the edges of
$I_{t+1},\ldots,I_m$ are located at the end of this ordering.
Now, we show that ${\rm ex}_{alt}(G,{\cal F},\sigma)={\rm ex}(G,{\cal F})$.
To see this, consider an alternating coloring of a subset of edges of
$G$ of length more than ${\rm ex}(G, {\cal F})$ with respect to the ordering $\sigma$.
In view of the ordering $\sigma$, for any $1\leq j\leq m$,
both colors are assigned to some edges of $I_j$, or exactly one edge of $I_j$ is colored, or no color
is assigned to any edge of $I_j$.
Denote the number of $j$'s such that both colors are assigned to some edges of $I_j$
by $k_1$, the number of $j$'s such that just red color is assigned to some edges of $I_j$
by $k_2$, and the number of $j$'s such that just blue color is assigned to some edges of $I_j$
by $k_3$. Without loss of generality, suppose that $k_2\geq k_3$.
Now, we have
$k_2+k_3> {\rm ex}(G, {\cal F})-rk_1=r.{\rm ex}(H,{\cal F})-rk_1\geq 2({\rm ex}(H,{\cal F})-k_1)$.
Therefore, $k_1+k_2>{\rm ex}(H,{\cal F})$ and so the red subgraph contains
some member of ${\cal F}$, which completes the proof.

To prove the second part,
in view of Lemma~\ref{ex}, it is sufficient to show that ${\rm ex}_{salt}(G, {\cal F})=r.{\rm ex}(H, {\cal F})+1$.
To see this, consider an alternating coloring of a subset of edges of
$G$ of length more than ${\rm ex}(G, {\cal F})+1$ with respect to the aforementioned ordering $\sigma$.
Consider $k_1$, $k_2$, and $k_3$, as defined in the previous part.
Since $r$ is an even integer,
we have $${r\over 2}k_1+k_2\geq {r.{\rm ex}(H,{\cal F})+2\over 2}={r\over 2}{\rm ex}(H,{\cal F})+1.$$
The proof is completed by showing that $k_1+k_2\geq {\rm ex}(H,{\cal F})+1$ and $k_1+k_3\geq {\rm ex}(H,{\cal F})+1$.
We just prove the first inequality and the same proof works for the other inequality.
On the contrary, suppose $k_1+k_2\leq {\rm ex}(H,{\cal F})$.
Therefore,
$${r\over 2}{\rm ex}(H,{\cal F})+1\leq {r\over 2}k_1+k_2\leq {r\over 2}({\rm ex}(H,{\cal F})-k_2)+k_2=k_2(1-{r\over 2})+{r\over 2}{\rm ex}(H,{\cal F})$$
which is a contradiction.
}\end{proof}
\subsection{Path Graphs}
We know that the general Kneser graph ${\rm KG}(C_n, P_d)$ is isomorphic to
the circular complete graph $K_{n\over d}$.
Hence, this motivates us to study
the chromatic number of the {\it path graph} ${\rm KG}(G, P_d)$.
Also, in view of Lemma~\ref{lowup}, one can see that for any general Kneser graph ${\rm KG}(G,{\cal F})$, we have $|E(G)|-2  {\rm ex}(G, {\cal F})\leq
 \chi({\rm KG}(G,{\cal F})) \leq |E(G)|-{\rm ex}(G, {\cal F})$.
We introduced several families of graphs whose chromatic numbers
attain the lower or upper bound. Hence, it may be of interest to
present some general Kneser graphs whose chromatic numbers
lie strictly between the lower bound and the upper bound. Now, by
determining the chromatic number of the path graph ${\rm KG}(G, P_{2})$,
we show that this graph
has such a property provided that $G$ is a dense graph.
\begin{lem}\label{countk12}
If a graph $G$ has $n$ vertices and $e$ edges,
then it has at least ${2e\over n}(e-{n\over 2})$ subgraphs each isomorphic to
the path $P_2$.
\end{lem}
\begin{proof}{
It is straightforward to check that the number of  subgraphs of $G$ isomorphic to $P_2$ is exactly
$\sum_{i=1}^{n}{{\rm deg}(v_i)\choose 2}$.
In view of Jensen's inequality, we have
$\sum_{i=1}^{n}{{\rm deg}(v_i)\choose 2}\geq n{{\sum {\rm deg}(v_i)\over n}\choose 2}$
which completes the proof.
}\end{proof}

Let $F$ be a subgraph of a graph $G$. We say $G$ has an {\it $F$-factor} if there are vertex-disjoint subgraphs $H_1,H_2,\ldots,H_t$ of $G$ such that each $H_i$ is isomorphic to ${\cal F}$ and $\displaystyle\bigcup_{i=1}^t V(H_i)=V(G)$.
Also, an independent set $S$ of general Kneser graph
${\rm KG}(G,{\cal F})$ is called {\it intersecting} independent set,
if there is an edge $e$ of $G$ appeared in each member of $S$.
Otherwise, it is termed a {\it non-intersecting} independent set of ${\rm KG}(G,{\cal F})$.
One can easily check that a non-intersecting
independent set in $G(n,P_2)$ has at most three members.
\begin{thm}\label{pathp2}
Let $G$ be a graph with $n$ vertices.
If $G$ has a spanning subgraph whose
connected components are $H_1,\ldots,H_p, H_{p+1}$, where
for any $1\leq i\leq p-1$, $H_i$ is a triangle and $H_{p}, H_{p+1}\in\{K_2,K_3\}$, then
$\chi({\rm KG}(G,P_2))=|E(G)|-\lfloor {2\over 3}n\rfloor$.
\end{thm}
\begin{proof}{
Set
$$T=\{e_1,e_2,\ldots,e_l\}=E(G)\setminus\displaystyle \displaystyle\cup_{i=1}^{p+1} E(H_i).$$
One can check that any connected subgraph of $G$ either has a nonempty intersection
with $T$ or is a subgraph of some $H_i$, for $1\leq i\leq p+1$. Also, one can see that if $P_2$ is a subgraph of $H_p$ (resp. $H_{p+1}$), then $H_p=K_3$ (resp. $H_{p+1}=K_3$). First, we show that there exists a proper coloring for ${\rm KG}(G,P_2)$
using $|E(G)|-\lfloor {2n\over 3}\rfloor$ colors.
If a subgraph $H$ of $G$ isomorphic to $P_2$ is a subgraph of $H_i$, then assign the color $l+i$ to $H$, where $l=|T|$.
Otherwise, let $j$ be the least integer such that $e_j\in E(H)$
and assign the color $j$ to $H$.
One can check that this coloring is a proper coloring
with $|E(G)|-\lfloor {2\over 3}n\rfloor$ colors.

Suppose that there exists a proper coloring of ${\rm KG}(G,P_2)$ with $\chi({\rm KG}(G,P_2))$
colors which has $t$ intersecting color classes and $s$ non-intersecting
color classes.
If $t\geq |E(G)|-\lfloor {2\over 3}n\rfloor$, then there is nothing to prove.
Therefore, suppose that $t< |E(G)|-\lfloor{2\over 3}n\rfloor$.
For each intersecting color class, remove an edge which appears in each member of this class to obtain
$\bar{G}$.
The graph $\bar{G}$ has $|E(G)|-t$ edges; and therefore, in view of Lemma~\ref{countk12}, it has at least
${2\left(|E(G)|-t\right)\left(|E(G)|-t-{n\over 2}\right)\over n}$
subgraphs isomorphic to $P_2$. Since every non-intersecting class has at most $3$ members, we have
$${2\left(|E(G)|-t\right)\left(|E(G)|-t-{n\over 2}\right)\over 3n}+t\leq\chi({\rm KG}(G,P_2)).$$
Now, set $x=|E(G)|-t> \lfloor{2\over 3}n\rfloor$.
Thus, we have $p(x)={2\over 3n}x(x-{n\over2})+|E(G)|-x\leq \chi({\rm KG}(G,P_2))$.
One can check that $p(x)$ takes its minimum in $x=n$, which is $|E(G)|-{2\over3}n$.
}\end{proof}

Hajnal and Szemeredi \cite{MR0297607} showed that for a graph $G$ if its minimum degree is at least
$(1-{1\over r})n$, then $G$ contains $\lfloor{n\over r}\rfloor$ vertex-disjoint copies of $K_r$.
 Corr{\'a}di and Hajnal \cite{MR0200185} investigated the maximum
number of vertex-disjoint cycles in a graph.
They showed that if $G$ is a graph of order at least $3k$ with minimum degree at least $2k$,
then $G$ contains $k$ vertex-disjoint cycles.
In particular, when the order of $G$ is exactly $3k$, then $G$ contains $k$ vertex-disjoint
triangles. Next, this result was extended as follows.
\begin{alphthm} {\rm \cite{MR1703268}} \label{discycle}
Let $G$ be a graph of order at least $3k$, where $k$ is a positive integer.
If for any pair of nonadjacent vertices $x$ and $y$
of $G$, we have ${\rm deg}_G(x)+{\rm deg}_G(y)\geq 4k - 1$, then $G$ contains $k$ vertex-disjoint cycles.
\end{alphthm}
Let $G$ be a graph such that ${\rm deg}_G(x)+{\rm deg}_G(y)\geq 4k-1$
for any pair of nonadjacent vertices $x$ and $y$ of $G$.
In the previous theorem, if we set $k=\lfloor{n\over 3}\rfloor$, then $n=3k+r$, where $0\leq r\leq 2$.
If $r=0$, then $G$ has $k$ vertex-disjoint cycles. Obviously, these cycles are triangles.

\begin{cor}
Let $G$ be a graph with $n$ vertices. If for any two nonadjacent vertices $x$ and $y$,
we have ${\rm deg}_G(x) + {\rm deg}_G(y)\geq {4n\over 3} - 1$, then $\chi({\rm KG}(G,P_2))=|E(G)|-\lfloor {2\over 3}n\rfloor$.
\end{cor}
\begin{proof}{
We show that $G$ has vertex-disjoint
subgraphs $H_1,H_2,\ldots,H_p$ such that $V(G)=\displaystyle\cup_{i=1}^tV(H_i)$
and for any $1\leq i\leq p-1$, $H_i$ is a triangle and $H_p\in\{K_2,K_3,2K_2\}$.

Assume that $n=3k+r$, where $0\leq r\leq 2$.
In view of Theorem~\ref{discycle}, $G$ has $k$ vertex-disjoint cycles.
We consider three different cases.
\begin{itemize}
\item[Case 1:]
If $r=0$, then $n=3k$ and $G$ has $p=k$ vertex-disjoint cycles $H_1,H_2,\ldots,H_p$
and all of these cycles should be triangles. Also, clearly we have $V(G)=\displaystyle\cup_{i=1}^tV(H_i)$.
\item[Case 2:]
For $r=1$, the graph $G$ has $p=k$ vertex-disjoint cycles where
$k-1$ of these cycles, say $H_1,H_2,\ldots,H_{k-1}$,
are triangles and the other cycle, say $C$, can be a $C_4$ or a triangle.

If $C$ is a
$C_4$, then remove two nonadjacent edges of it to obtain $H_{k}\cong 2K_2$.
Otherwise, if $C$ is a triangle, then assume that $z$ is a vertex that is not in
$V(C)\cup(\displaystyle\cup_{i=1}^{k-1}V(H_i))$.
Consider $u\in V(G)$ such that $uz\in E(G)$. Without loss of generality,
we can suppose that $V(C)=\{u,v,w\}$.
Now, consider $H_k\cong 2K_2$ with the vertex set $\{z,u,v,w\}$ and the edge set $\{zu,vw\}$.
One can check that $V(G)=\displaystyle\cup_{i=1}^tV(H_i)$.

\item[Case 3:]
If $r=2$, we add a new vertex $z$ to $G$ and join it to all vertices of $G$
to obtain the graph  $G'$.
The graph $G'$ has $3(k+1)$ vertices and for any two nonadjacent vertices $x$ and $y$,
${\rm deg}_{G'}(x)+{\rm deg}_{G'}(y)\geq 4{n\over 3} - 1+2\geq 4(k+1)-1$.
Therefore, by Theorem~\ref{discycle}, $G'$ has $k+1=p$ vertex-disjoin triangles.
By removing the vertex $z$, $G$ has a spanning subgraph $H$ such that
it has $k+1$ connected components where $k$ of them are
triangles and one of them is $K_2$.
\end{itemize}
Now,
in view of Theorem~\ref{pathp2}, the assertion follows.
}
\end{proof}

The degree condition mentioned in the aforementioned corollary cannot be dropped. To see this, one can check that
$\chi({\rm KG}(C_5,P_2)))=3\not =|E(C_5)|-\lfloor {10\over 3}\rfloor$.
\subsection{Concluding Remarks}
In Lemma~\ref{lowup}, we introduced a lower and upper bound for the chromatic number of graphs
in terms of the generalized Tur\'an number. It can be of interest
to find necessary and sufficient conditions to know when equality holds in both cases. In this regard, in view of these bounds,
one can reformulate several interesting results or conjectures. Here, we present some of them.

\begin{enumerate}
\item The Kneser graph ${\rm KG}(nK_2, kK_2)$ (Lov{\'a}sz~\cite{MR514625}):
$$\chi({\rm KG}(nK_2, kK_2))=|E(nK_2)|-2  {\rm ex}(nK_2, kK_2)=n-2k+2.$$

\item The Schrijver graph ${\rm KG}(C_n, kK_2)$ (Schrijver~\cite{MR512648}):
$$\chi({\rm KG}(C_n, kK_2))=|E(C_n)|-{\rm ex}(C_n, kK_2)=n-2k+2.$$

\item The Kneser hypergraph graph ${\rm KG}^r(nK_2, kK_2)$ (Alon, Frank, and Lov{\'a}sz~\cite{MR857448}):
$$\chi({\rm KG}^r(nK_2, kK_2))=\left\lceil{|E(nK_2)|-r.{\rm ex}(nK_2, kK_2)\over r-1}\right\rceil
=\left\lceil{n-r(k-1)\over r-1}\right\rceil.$$

\item The generalized Kneser graph ${\rm KG}(K_n,K_k)$ (Frankl~\cite{MR797510}):
$$\chi({\rm KG}(K_n,K_k))=|E(K_n)|-{\rm ex}(K_n,K_k)=
(k-1){s \choose 2}+rs,$$
where $n=(k-1)s+r$, $0\leq r < k-1$, and $n$ is sufficiently large.

\end{enumerate}

We should mention that the chromatic number of the generalized Kneser graph ${\rm KG}(K_n,K_3)$ (i.e., ${\rm KG}(n,3,1)$) was determined by Tort~\cite{MR689810}. Also, Frankl~\cite{MR797510} introduced the
following conjecture about the chromatic number of generalized Kneser graphs.

\begin{alphcon}{\rm (Frankl~\cite{MR797510})}
Let $n, k$ and $s$ be positive integers, where $k>s\geq 2$ and $n\geq 2k-s+1$. If $n$ is sufficiently large, then
$$\chi({\rm KG}(K_n^s,K_k^s))=|E(K_n^s)|-{\rm ex}(K_n^s, K_k^s),$$
where the complete hypergraph $K_n^s$ consists of all $s$-subsets of $[n]$.
\end{alphcon}

Also, in \cite{MR3073134}, several conjectures and problems are introduced.
Again, these problems can been reformulated in terms of the generalized Tur\'an number as follows.

\begin{alphcon}\label{oddcy}{\rm \cite{MR3073134}}
If $k$ is an odd integer and $n$ is sufficiently large, then
$$\chi({\rm KG}(K_n,C_k))=|E(K_n)|-{\rm ex}(K_n,C_n)=\lfloor {(n-1)^2 \over 4}\rfloor .$$
\end{alphcon}

\begin{alphprb}\label{evency}{\rm \cite{MR3073134}}
Let $k$ be an even integer. Does
$${n\choose 2}-O(n^{1+{2\over k}}) \leq \chi({\rm KG}(K_n,C_k))\leq {n\choose 2}-\Omega(n^{1+{1\over k}})$$
hold{\rm ?}
\end{alphprb}

\begin{alphprb}\label{fourcy}{\rm \cite{MR3073134}}
Is the following statement true{\rm ?}
$$\chi({\rm KG}(K_n,C_4))= {n\choose 2}-{1\over 2}n^{3\over 2}+o(n^{3\over 2}).$$
\end{alphprb}

It is known that ${\rm ex}(K_n,C_4)={1\over 2}n^{3\over 2}-o(n^{3\over 2})$. Hence, it may be of interest to know whether the equality $\chi({\rm KG}(K_n,C_4))=|E(K_n)|-{\rm ex}(K_n,C_4)$ holds provided that $n$ is sufficiently large.

\begin{alphprb}\label{cyprime}{\rm \cite{MR3073134}}
If $q$ is a prime power and $n=q^2 +q+1$, does
$$\chi({\rm KG}(K_n,C_4))= |E(K_n)|-{\rm ex}(K_n,C_4)={q^2 +q+1\choose 2}-{1\over 2}q(q+1)^2$$
hold{\rm ?}
\end{alphprb}

It is known that
$\Omega(n^{1+{1\over k}}) \leq {\rm ex}(K_n,C_k)\leq O(n^{1+{2\over k}})$, so in view of Lemma~\ref{lowup} we have the following proposition which gives an affirmative answer to Problem~\ref{evency}.
\begin{pro}
Let $k$ be an even integer. We have
$${n\choose 2}-O(n^{1+{2\over k}}) \leq \chi({\rm KG}(K_n,C_k))\leq {n\choose 2}-\Omega(n^{1+{1\over k}}).$$
\end{pro}

We can generalize the definition of Kneser representation by considering
labeled hypergraphs, i.e.,  we assign some labels to the hyperedges or vertices of $H$ and ${\cal F}$.
In this terminology, the hypergraph
${H\choose {\cal F}}$
has $E(H)$ as its vertex set and each
subhypergraph of $H$ isomorphic to a member of ${\cal F}$ forms a hyperedge, where
any isomorphism should preserve the labels. One can define the hypergraphs ${\rm KG}^r(H, {\cal F})$
similar to unlabeled ones. This new definition helps us to introduce some appropriate representations
for some families of graphs. For instance, here we consider the Cartesian sum of Kneser graphs.
{\it The Cartesian sum} $G\oplus H$ of two graphs $G$ and $H$ has the vertex set $V(G)\times V(H)$ and two vertices
$(u,v)$ and $(a,b)$ are adjacent, if either $u$ is adjacent to $a$ or $v$ is adjacent to $b$. The chromatic number of Cartesian sum of
graphs has been studied in several papers, see~\cite{MR2464883, MR0223273}. In general, it seems that it is~not easy to evaluate
the chromatic number of the Cartesian sum of two graphs. One can check that $\chi(G\oplus H) \leq \chi(H)\chi(G)$.
In \cite{MR2464883}, this bound
was improved and also the chromatic number of the Cartesian sum of circular complete graphs was determined.
As a natural question, it can be of interest to know the chromatic number of Cartesian sum of Kneser graphs. In particular, we
are interested in finding the chromatic number of $\chi({\rm KG}(m,2)\oplus {\rm KG}(n,2))$. Note that ${\rm KG}(m,p)\oplus {\rm KG}(n,p)$
is isomorphic to ${\rm KG}(K_{m,n},K_{p,p})$. The problem of finding the chromatic number of ${\rm KG}(K_{m,n},K_{2,2})$
can be considered as a twin of Problem~\ref{fourcy}. In view of Lemma~\ref{lowup}, we have
$$ mn-2  {\rm ex}(K_{m,n},K_{p,p}) \leq \chi({\rm KG}(m,p)\oplus {\rm KG}(n,p)) \leq mn-{\rm ex}(K_{m,n},K_{p,p}).$$
One can see that the graph
${\rm KG}(\vec{K}_{m,n},\vec{K}_{p,q})$ is isomorphic to ${\rm KG}(m,p)\oplus {\rm KG}(n,q)$, where
the labeled graph $\vec{K}_{m,n}$ is obtained from the complete bipartite graph $K_{m,n}$ by assigning the label
one to the vertices of the part with size $m$ and the label two to the others (or by assigning the same direction to all edges
form the part of size $m$ to the other part). Note that if $p\not = q$ and $\min\{m, n\}\geq \max\{p, q\}$, then  the graph ${\rm KG}(K_{m,n},K_{p,q})$ is~not
isomorphic to the graph ${\rm KG}(\vec{K}_{m,n},\vec{K}_{p,q})$.

\begin{qu}\label{Kneser}
Assume that $m, n, p,$ and $q$ are positive integers.  What are the values
of $\chi({\rm KG}(K_{m,n},K_{p,q}))$ and $\chi({\rm KG}(\vec{K}_{m,n},\vec{K}_{p,q}))$ {\rm ?}
\end{qu}
{\bf Acknowledgement:} The authors gratefully acknowledge for many stimulating conversations and the many helpful suggestions of Professor Carsten~Thomassen during the preparation of the paper.
Also, they wish to thank Dr.~Saeed~Shaebani for his useful comments. A part of this paper was written while Hossein Hajiabolhassan was visiting School of Mathematics, Institute for Research in Fundamental Sciences~(IPM). He acknowledges the support of IPM (No. $94050128$).
Furthermore, the authors would like to thank Skype for sponsoring their endless conversations in two countries.

\def\cprime{$'$} \def\cprime{$'$}


\end{document}